\documentclass[11pt]{article}

\usepackage{amsmath}
\usepackage{amssymb}
\usepackage{amsthm}
\usepackage{graphicx}

\newcommand{\minus}{\smallsetminus}
\newcommand{\N}{{\mathbb{N}}}
\newcommand{\Z}{{\mathbb{Z}}}
\newcommand {\ignore}[1]  {}
\renewcommand{\S}{S_{\infty}}
\newcommand{\sym}[1]{\text{Sym}(#1)}
\renewcommand{\subset}{\subseteq}

\renewcommand{\phi}{\varphi}
\newcommand{\tuple}[2]{({#1}_1,\dots,{#1}_{#2})}
\newcommand{\gen}[1]{\langle{#1}\rangle}
\newcommand{\F}{\mathcal{F}}
\newcommand{\tr}[2]{\mathrm{tr}_{#2}(#1)}

\newtheorem{thm}{Theorem}[section]
\theoremstyle{definition}
\newtheorem{definition}[thm]{Definition}
\theoremstyle{plain}
\newtheorem{lem}[thm]{Lemma}
\newtheorem{prop}[thm]{Proposition}
\newtheorem*{main}{Main Theorem}
\theoremstyle{remark}
\newtheorem*{notation}{Notation}

\begin{document}
\title{Highly Transitive Actions of Surface Groups}
\author{Daniel Kitroser}
\maketitle

\begin{abstract}
A group action is said to be highly-transitive if it is $k$-transitive for every $k\geq 1$. The main result of this thesis is the following:
\begin{main}
The fundamental group of a closed, orientable surface of genus $> 1$ admits a faithfull, highly-transitive action on a countably infinite set.
\end{main}
From a topological point of view, finding a faithfull, highly-transitive action of a surface group is equivalent to finding an embedding of the surface group into $\sym{\Z}$ with a dense image. In this topological setting, we use methods originally developed in \cite{BGSS06} and \cite{BG09} for densely embedding surface groups in locally compact groups. 
\end{abstract}

\section{Introduction}

A permutation group $G\leq \text{Sym}(X)$ is called \emph{$k$-transitive} if it is transitive on ordered $k$-tuples of distinct elements and \emph{highly-transitive} if it is $k$-transitive for every $k \in \N$. When $G$ is given as an automorphism group, it can often be verified that the given action is highly transitive. This is the case for any subgroup of $\text{Sym}(X)$ which contains all finitely supported permutations, where $X$ is any infinite set. Other examples include groups such as $\operatorname{Homeo}(S^2)$ and many of its subgroups. But for most groups the question: ``what is the maximal $k$ for which the group admits a faithfull, $k$-transitive action''? is wide open. 

It was shown by McDonough \cite{M77} that any non-abelian free group admits a highly transitive action. Another proof which is more useful from our point of view was later given by Dixon in \cite{D90}, using a Baire Category type argument. 

In this paper, we prove that a \emph{surface group} (the fundamental group of a closed, orientable surface without boundry of genus greater than $1$) admits a faithfull, highly-transitive action of countably infinite degree. 

\subsection{Topological interpretation}

Before we continue, we introduce some notation.
\begin{notation}
We denote by $\S=\sym{\Z}$ the full symmetric group of a countable set which we identify with the integer numbers. The action of $G\leq \S$ on $\Z$ is always from the right and the image of an element $a\in\Z$ under $g\in G$ is denoted by $a^g$. If $n\in\N$ then $[n]$ will denote the set of all integers between $-n$ and $n$.
\end{notation}

We define a group topology on $\S$ where a basis is given by the sets of the form 
\[
U(\phi,[n])=\left\{\psi\in\S\ \left|\ \psi\big|_{[n]} =  \phi\big|_{[n]}\right. \right\},\ (\phi\in\S,n\in\N).
\]
It is easy to see that if $\phi\in\S$ and $A\subset\Z$ is any finite set then $U(\phi,A)$ is open in this topology. Notice that this topology is the    restriction to $\S$ of the topology of pointwise convergence on the space of all functions $\Z\to \Z$ and a straightforward proof shows that $\S$ is a topological group which is completely metrizable.

The problem of finding highly-transitive subgroups of $\S$ can now be approached via the following proposition, derived directly from the definition of the topology on $\S$.
\begin{prop}\label{HT=dense}
Let $G\leq\S$. Then $G$ is highly-transitive if and only if $G$ is dense in $\S$. 
\end{prop}

\subsection{The main result}

In view of proposition \ref{HT=dense}, the main result can be reformulated as follows. 
\begin{thm}[main theorem]\label{main}
Let $\Gamma = \pi_1(\Sigma_g)$ be the fundamental group of an orientable surface of genus $\geq 2$. Then there exists a dense subgroup of $\S$ which is isomorphic to $\Gamma$. 
\end{thm}

The methods used in this work to obtain dense embeddings of surface groups into $\S$ are analogous to those used in \cite{BGSS06} and \cite{BG09} to show that if a locally compact group contains a dense, free subgroup of every finite rank $> 1$ then it contains a dense surface group of every genus $>1$. $\S$ is \emph{not} locally-compact but as we shall see, by prooving that $\S$ has a dense free subgroup of every finite rank $> 1$ with the additional property that certain elements of that free group generate a non-discrete cyclic group we can apply the same methods in our setting.

\section{Dense Embeddings of Free Groups}

As a first step to proving theorem \ref{main} we prove the following.

\begin{thm}\label{free-dense}
Let $r\geq 2$, let $w_i = w_i\tuple{\tau}{r},\ (i\in\N)$ be reduced words and assume that there exists $j\in\{1,\dots,r\}$ such that every $w_i$ is not a conjugate of a power of $\tau_j$. Then, there exist $\tau_1,\dots,\tau_r\in\S$ such that $F=\gen{\tau_1,\dots,\tau_r}$ is a dense, rank $r$ free subgroup of $\S$ and such that $\gen{w_i\tuple{\tau}{r}}$ is non-discrete for every $i\in\N$.
\end{thm}

We prove the theorem using Baire's Category Theorem. 

\begin{definition}
Let $X$ be a topological space. A subset of $X$ which is a countable intersection of dense, open sets is called \emph{residual} or \emph{generic}.
\end{definition}

Baire's theorem states that in a non-trivial, complete metric space, a residual set is dense. Since they are dense and closed under countable intersections, residual sets in a complete metric space can be thought of as being ``large''. Dixon showed in \cite{D90} that if  we denote 
\[U = \{\tuple{\tau}{r}\in\S^r\ |\ \gen{\tau_1,\dots,\tau_r}\text{ is dense in $\S$}\}\]
then the set of elements in $U$ that freely generate a free group is residual in $\overline{U}$. We use a somewhat different setting than Dixon in order to prove the existence of the desired dense free subgroups. Fix $\sigma\in\S$ to be the \emph{shift} permutation, i.e. $\forall a\in\Z : a^{\sigma} = a+1$. We show that for every $n\geq 1$, the set of elements $\tuple{\tau}{n}\in\S^n$ such that $\gen{\sigma,\tau_1,\dots,\tau_n}$ has the properties stated in theorem \ref{free-dense} is residual is $\S^n$.

\begin{lem}\label{nondisc}
Let $\gamma\in\S$. $\gen{\gamma}\leq\S$ is non-discrete if and only if the orbits of $\gen{\gamma}$ are all finite and of unbounded length.
\end{lem}
\begin{proof}
Assume first that $\gen{\gamma}$ has an infinite orbit $\Delta$ and let $a\in\Delta$. Then, $\gen{\gamma}\cap U(1,\{a\}) = \{1\}$ and so $\gen{\gamma}$ is discrete. Now assume that the length of the orbits of $\gen{\gamma}$ is bounded. Let $m$ be the product of the lengths of the orbits of $\gen{\gamma}$. So, $\gamma^m = 1$ thus, $\gen{\gamma}$ is finite and hence discrete.  

Conversely, suppose that the orbits of $\gen{\gamma}$ finite and of unbounded length. To prove that $\gen{\gamma}$ is non-discrete it is enough to show that every basic neighborhood of the identity contains some non-trivial power of $\gamma$. Let $n\in\N$ and let $\Delta = \cup_{i=1}^{k}\Delta_i$ be a finite union of orbits of $\gen{\gamma}$ such that $[n]\subset\Delta$. By hypothesis, all the $\Delta_i$'s are finite. If we set $m=\prod_{i=1}^{k}|\Delta_i|$ then for every $a\in\Delta$ (and in particular for every $a\in	[n]$) we have that $a^{\gamma^m}=a$ and since the orbit lengths of $\gen{\gamma}$ are unbounded, there is an orbit of $\gen{\gamma}$ which is longer then $m$ so $\gamma^m$ is not the identity element. Thus, $\gamma^m\neq 1$ is an element of the pointwise stablilizer $U(\text{id},[n])$. Since such stabilizers form a basis at the identity, we are finished.
\end{proof}

\begin{definition}
Let $\gamma_1,\dots,\gamma_n\in\S$ and let $w=w(\gamma_1,\dots,\gamma_n)$ be any word. If $w=w_1 w_2\cdots w_n$ where $w_i\in \{\gamma_1^{\pm 1},\dots,\gamma_n^{\pm 1} \}$ then the \emph{trace} of an element $a\in\Z$ under $w$ is the ordered set   
\[
	\tr{a}{w} = \{ a,a^{w_1},a^{w_1w_2},\dots,a^{w_1w_2\cdots w_n}=a^w\}. 
\]
\end{definition} 

\begin{lem}\label{residual}
	Fix $n\in\N$ and let $w = w(\sigma,\tau_1,\dots,\tau_n)$ be a reduced word which is not 
	conjugated to a power of $\sigma$. Then the following sets are residual:
	\begin{align*}
		&(1)& \F =\ & \{ (\tau_1,\dots,\tau_n)\in\S^n\ |\ 
		\gen{\sigma,\tau_1,\dots,\tau_n} \text{ is a free group of rank $n$ }\}	\\	
		&(2)&\mathcal{D} =\ & \{ (\tau_1,\dots,\tau_n)\in\S^{n}\ |\ 
		\gen{\sigma,\tau_1,\dots,\tau_n}\text{ is dense}\}\\ 
		&(3)&\mathcal{N} =\ & \{ (\tau_1,\dots,\tau_n)\in\S^{n}\ |\ 
		\gen{w(\sigma,\tau_1,\dots,\tau_n,)} \text{ is non-discrete}\} 
	\end{align*}
\end{lem}
\begin{proof}
\textbf{(1)} Let $v$ be a reduced, non-trivial word on $n+1$ letters and consider the set
\[
	\F_v = \{ (\tau_1,\dots,\tau_n)\in\S^n\ |\ v(\sigma,\tau_1,\dots,\tau_n)\neq 1 \}.
\]
If we prove that $\F_v$ is open and dense in $\S^n$ for every $v$ as above then $\F$ is residual since $\F =\bigcap_{v\neq 1} \F_v$. Obviously, $\F_v$ is open as the inverse image of the open set $\S\minus \{1\}$ under the continuous mapping $(\tau_1,\dots,\tau_n)\mapsto v(\sigma,\tau_1,\dots,\tau_n)$. To prove that $\F_v$ is dense let $(\phi_1,\dots,\phi_n)\in\S^n$ and let $m\in\N$. We prove that there exists $(\tau_1,\dots,\tau_n)\in\F_v$ such that $\tau_i\big|_{[m]}=\phi_i\big|_{[m]}$ for every $1\leq i\leq n$. First, write 
\[
	v(\sigma,\tau_1,\dots,\tau_n)=\sigma^{r_1}v_1\sigma^{r_2}v_2\cdots\sigma^{r_k}v_k\sigma^{r_{k+1}}
\]
where $r_i\in\Z$ and $v_i\in\{\tau_1^{\pm 1},\dots,\tau_n^{\pm 1}\}$ for every $i$. We want define the permutations $\tau_1,\dots,\tau_n$.
Choose $a_1,\dots,a_{k+1}\in\Z$ such that the numbers $a_1,\dots,a_{k+1},a_1+r_1,\dots,a_{k+1}+r_{k+1}$ are all distinct and define $(a_i+r_i)^{v_i}=a_{i+1}$ (e.g. if $v_1=\tau_5^{-1}$ define $(a_1+r_1)^{\tau_5^{-1}}=a_2$ or equivalently, define $a_2^{\tau_5}=a_1+r_1$). Note that since all the integers involved in the definition of the $v_i$'s are distinct and $v$ is reduced, the definitions of the $v_i$'s do not contradict each other. We now have that $a_1^v = a_{k+1}+r_{k+1}$ and in particular, $v$ is not trivial. In order to fulfill the condition that $\tau_i$ must send an element $b\in [m]$ to $b^{\phi_i}$ we choose $a_1,\dots,a_{k+1}$ in a way that will not contradict with this requirment. Explicitly, choose    
\[
	a_1,\dots,a_{k+1}\in \Z \minus \left( [m]\cup\bigg(\bigcup_i [m]^{\phi_i}\bigg) \cup\bigg(\bigcup_j 
	[m]-r_j\bigg)\cup\bigg(\bigcup_{i,j} [m]^{\phi_i}-r_j\bigg) \right).
\]
Note that we can always choose $a_1,\dots,a_{k+1}$ in the manner described since we are only excluding a finite set of integers that we can not choose from. Now, every $\tau_i$ is defined on $[m]$ and on some other elements $\{b_1,\dots,b_{\ell}\} \subset  
\{a_1,\dots,a_k,a_1+r_1,\dots,a_{k+1}+r_{k+1}\}$, sending them to $\{c_1,\dots,c_{\ell}\}$ respectively. Finally, choose any bijection
\[
	f_i: \Z\minus \big( [m]\cup\{b_1,\dots,b_{\ell}\}\big)\to\Z\minus\big([m]^{\phi_i}\cup\{c_1,\dots,c_{\ell}\}\big)
\]
and define
\[
	x^{\tau_i} =
	\begin{cases}
		x^{\phi_i}&,x\in[m].\\
		c_j&, x=b_j \text{ for some $j$}.\\
		x^{f_i}&,\text{else}.
	\end{cases}
\]
So we have that every $\tau_i$ is a permutation which lies in the basic neighbourhood of $\phi_i$ defined by $[m]$ such that $v(\sigma,\tau_1,\dots,\tau_n)$ is non-trivial.
\newline\newline
\textbf{(2)} Fix some $k\in\N$ and for every two $k$-tuples $\mathbf{x} = (x_1,\dots,x_k)$ and $\mathbf{y} = (y_1,\dots,y_k)$ of distinct integers consider the set
\[
	\mathcal{D}_{\mathbf{x},\mathbf{y}} = \left\{ (\tau_1,\dots,\tau_n)\in\S^n\ \left| \
	\begin{array}{l}
		\exists \phi\in\gen{\sigma,\tau_1,\dots,\tau_n}:\\
		\forall 1\leq j\leq k : x_j^{\phi} = y_j
	\end{array}\right. \right\}.
\] 
Notice that if $(\tau_1,\dots,\tau_n)\in\bigcap_{\mathbf{x},\mathbf{y}}\mathcal{D}_{\mathbf{x},\mathbf{y}}$ where $\mathbf{x}$ and $\mathbf{y}$ range over all $k$-tuples of distinct integers then $\gen{\sigma,\tau_1,\dots,\tau_n}$   
is $k$-transitive. Thus, by proposition \ref{HT=dense} we have that 
\[ \mathcal{D}=\bigcap_{k\in\N}\bigcap_{\mathbf{x},\mathbf{y}}\mathcal{D}_{\mathbf{x},\mathbf{y}}\] and we are left to show that $\mathcal{D}_{\mathbf{x},\mathbf{y}}$ is open and dense for every $\mathbf{x}$ and $\mathbf{y}$ as
above.

First, fix $\mathbf{x} = (x_1,\dots,x_k),\mathbf{y} = (y_1,\dots,y_k)$ as above and let $(\tau_1,\dots,\tau_n)\in\mathcal{D}_{\mathbf{x},\mathbf{y}}$. By definition there exists $\phi \in \gen{\sigma,\tau_1,\dots,\tau_n}$ such that $x_i^{\phi} = y_i$ for every $i=1,\dots,k$. Let us write $\phi = v(\sigma,\tau_1,\dots,\tau_n)$ where $v=v(\sigma,\tau_1,\dots,\tau_n)$ is a word on $\{\sigma^{\pm 1},\tau_1^{\pm 1},\dots,\tau_n^{\pm 1}\}$ and define $A=\bigcup_{i=1}^k \tr{x_i}{v}$. $A$ is finite and so the set
\[
	\mathcal{U} = \left\{ (\psi_1,\dots,\psi_n)\in\S^n\ \left| \
	\begin{array}{l}
		\psi_i\big|_A = \tau_i\big|_A,\ \psi_i^{-1}\big|_A = \tau_i^{-1}\big|_A\\
		i=1,\dots,n
	\end{array} \right. \right\}
\]
is an open neighbourhood of $(\tau_1,\dots,\tau_n)$ contained in $\mathcal{D}_{\mathbf{x},\mathbf{y}}$. Indeed, if $(\psi_1,\dots,\psi_n)\in\mathcal{U}$ take $\xi=v(\sigma,\psi_1,\dots,\psi_n)\in\gen{\sigma,\psi_1,\dots,\psi_n}$, then by the definition of $\mathcal{U}$, we have that $\xi$ acts the same as $\phi=v(\sigma,\tau_1,\dots,\tau_n)$ on $x_i$ and in particular, $\xi$ sends $x_i$ to $y_i$. This shows that $\mathcal{D}_{\mathbf{x},\mathbf{y}}$ is open.

Now we prove that $\mathcal{D}_{\mathbf{x},\mathbf{y}}$ is dense. Let $(\phi_1,\dots,\phi_n)\in\S^n$ and $m\in\N$. Let $r\in\N$ be such that $x_j+r\notin [m]$ and $y_j+r\notin [m]^{\phi_1}$ for every $1\leq j\leq k$. Let 
\[
	f:\Z\minus\big( [m]\cup\{x_1+r,\dots,x_k+r\}\big) \to \Z\minus\big( [m]^{\phi_1}\cup\{y_1+r,\dots,y_k+r\}\big)
\]
be any bijection and define
\[
	x^{\tau_1} = 
	\begin{cases}
		x^{\phi_1}&, x\in [m]\\
		y_j+r&, x=x_j+r \text{ for some $1\leq j\leq k$}\\
		x^f&,\text{otherwise}
	\end{cases}.
\]
Now define $\tau_i=\phi_i$ for every $2\leq i\leq n$ and we get that $\tuple{\tau}{n}$ is an element of the basic neighbourhood of $\tuple{\phi}{n}$ defined by $[m]$. Also, the permutation $\xi=\sigma^r \tau_1\sigma^{-r}\in\gen{\sigma,\tau_1,\dots,\tau_n}$ sends each $x_j$ to $y_j$ thus, $\tuple{\tau}{n}\in \mathcal{D}_{\mathbf{x},\mathbf{y}}$.   
\newline\newline
\textbf{(3)} By lemma \ref{nondisc} we can equivalently write
\[
	\mathcal{N}=\left\{ \tuple{\tau}{n}\in\S^n\ \left|\
	\begin{array}{l}
 		\text{The orbits of $\gen{w(\sigma,\tau_1,\dots,\tau_n)}$ are all finite}\\
		\text{and of unbounded length.}
	\end{array} \right. \right\}.  
\]
Thus, if we define for every $t\in\N$ and $a\in\Z$:
\begin{align*}
	\mathcal{U}_t &= \{ \tuple{\tau}{n}\in\S^n\ |\ \gen{w(\sigma,\tau_1,\dots,\tau_n)} \text{ has an orbit of 
	length $\geq t$}\}.\\
	\mathcal{V}_a &= \{ \tuple{\tau}{n}\in\S^n\ |\ \text{The orbit of $a$ under 
	$\gen{w(\sigma,\tau_1,\dots,\tau_n)}$ 
	is finite} \}.
\end{align*}
we get that $\mathcal{N}=\big(\bigcap_{t\in\N}\mathcal{U}_t\big) \cap \big( \bigcap_{a\in\Z} \mathcal{V}_a\big)$.

We now show that $\mathcal{U}_t$ and $\mathcal{V}_a$ are open and dense for every $t\in\N$ and $a\in\Z$.
Let $\tuple{\tau}{n}\in\mathcal{U}_t$ and let $b\in\Z$ be an element belonging to an orbit of  $\gen{w(\sigma,\tau_1,\dots,\tau_n)}$ of length $\geq t$. Thus, $b,b^w,b^{w^2},\dots,b^{w^{t-1}}$ are all distinct. 
Let $\Delta = \bigcup_{i=1}^{k-1} \tr{b}{w^{i}}$. The set
\[
	\left\{ \tuple{\psi}{n}\in\S^n\ \left| 
	\begin{array}{l}
		\psi_i\big|_{\Delta} = \tau_i\big|_{\Delta},\ \psi_i^{-1}\big|_{\Delta} = \tau_i^{-1}\big|_{\Delta}\\
		i=1,\dots,n
	\end{array} \right. \right\}
\]
is an open neighbourhood of $\tuple{\tau}{n}$ which is contained in $\mathcal{U}_t$ hence, $\mathcal{U}_t$ is open.

Now, take $\tuple{\tau}{n}\in\mathcal{V}_a$ and let $\Delta = \{a_1,\dots,a_s\}$ be the finite orbit of $\gen{w(\sigma,\tau_1,\dots,\tau_n)}$ containing $a$. Similarly
\[
	\left\{ \tuple{\psi}{n}\in\S^n\ \left| 
	\begin{array}{l}
		\psi_i\big|_{\Delta} = \tau_i\big|_{\Delta},\ \psi_i^{-1}\big|_{\Delta} = \tau_i^{-1}\big|_{\Delta}\\
		i=1,\dots,n
	\end{array} \right. \right\}
\]
is an open neighbourhood of $\tuple{\tau}{n}$ which is contained in $\mathcal{V}_a$.

To prove $\mathcal{U}_t$ is dense, let $\tuple{\phi}{n}\in\S^n$ and $m\in\Z$. In (1) we in fact showed that we can define $\tau_1,\dots,\tau_n\in\S$ such that for every finite set $A\subset\Z$ we have that $\tau_i\big|_A$ acts in any way we please and there exists some $b\in\Z$ such that $b^w\neq b$ where $w=w(\sigma,\tau_1,\dots,\tau_n)$. By repeating the same argument we can find $\tau_1,\dots,\tau_n\in\S$ such that $\tau_i\big|_{[m]} = \phi_i\big|_{[m]}$ for every $1\leq i\leq n$ and such that there exists $b\in\Z$ such that $b,b^w,b^{w^2},\dots,b^{w^{t-1}}$ are all distinct i.e. $\gen{w(\sigma,\tau_1,\dots,\tau_n)}$ has an orbit of length $\geq t$.

Finally, we prove that $\mathcal{V}_a$ is dense. Since $\gen{w(\sigma,\tau_1,\dots,\tau_n)}$ has the same orbit structure as the cyclic group generated by any conjugate of $w$, we can assume without loss of generality that $w$ is a cyclically reduced word that is not a power of $\sigma$. Let $\tuple{\phi}{n}\in\S^n$ and $m\in\Z$. We need to define permutations $\tau_1,\dots,\tau_n\in\S$ such that $\tuple{\tau}{n}\in\mathcal{V}_a$ and every $\tau_i$ agrees with $\phi_i$ on $[m]$. This condition can be thought of in the following way: $\tau_i$ is already defined on $[m]$ for every $i$ and $\tau_i^{-1}$ is already defined on $[m]^{\phi_i}$ for every $i$ (they act the same as $\phi_i$ and $\phi_i^{-1}$ respectively) and we are left to define $\tau_i$ on $\Z\minus [m]$ (and $\tau_i^{-1}$ on $\Z\minus [m]^{\phi_i}$) in such a way that the orbit of $a$ under  $\gen{w(\sigma,\tau_1,\dots,\tau_n)}$ will be finite. First we write $w(\sigma,\tau_1,\dots,\tau_n)=w_1w_2\cdots w_k$. Now, we start by applying the positive and negative powers of $w$ to $a$ letter by letter:
\[
	\xrightarrow{w_\ell}c\xrightarrow{w_{\ell+1}} \dots \xrightarrow{w_{n-1}} a_{-1} \xrightarrow{w_n}
	a=a_0 \xrightarrow{w_1} a_1 \xrightarrow{w_2} \dots \xrightarrow{w_{s-1}} b \xrightarrow{w_s}
\]
where $b\in\Z$ is the first element such that we need to apply to it the permutation $w_s$ and $w_s$ is not yet defined on $b$, that is, $w_s=\tau_i$ for some $i$ and $b\notin [m]$ or $w_s=\tau_i^{-1}$ for some $i$ and $b\notin [m]^{\phi_i}$. Note that if such an element $b$ does not exist then, since by hypothesis at least one of the letters $w_1,\dots,w_k$ is not $\sigma$, the orbit of $a$ under $\gen{w(\sigma,\phi_1,\dots,\phi_n)}$ is contained in $[m]\cup\big(\bigcup_{i=1}^n [m]^{\phi_i}\big)$, hence finite and we can just take $\tau_i=\phi_i$ for every $i$. We can thus assume that such a $b$ exists. Similarly, $c\in\Z$ is the first element we reach when applying letter by letter the negative powers of $w$ to $a$ such that we need to apply to $c$ the permutation $w_{\ell}^{-1}$ and $w_{\ell}^{-1}$ is not yet defined on $c$. As before, we can assume without loss of generality that such an element $c$ exists. By hypothesis, $w$ is cyclically reduced and so the word $w_sw_{s+1}\cdots w_kw_1\cdots w_{\ell}$ is reduced. Also, by their definition, $w_s,w_\ell\neq\sigma^{\pm 1}$. We wish to define $\tau_1,\dots,\tau_n$ in such a way that $b^{w_sw_{s+1}\cdots w_kw_1\cdots w_{\ell}}=c$ (and of course fulfilling the condition that $\tau_i$ agrees with $\phi_i$ on $[m]$). By repeating the argument made in (1), we can find two distinct elements $d_1,d_2\in\Z$ that do not lie in $[m]$ or any $[m]^{\phi_i}$ such that $w_{s+1} \cdots w_k w_1 \cdots w_{\ell-1}$ sends $d_1$ to $d_2$ and define $b^{w_s}=d_1,\ d_2^{w_\ell} = c$. From this we get that $b^{w_s \cdots w_k w_1 \cdots w_l} = d_1^{w_{s+1} \cdots w_k w_1 \cdots w_\ell} = d_2^{w_\ell} = c$ thus, the orbit of $a$ is finite.  Now, every $\tau_i$ is defined on $[m]$ and maybe on finitely many more elements. Again, exactly as we did in (1), we can extend the definition of $\tau_i$ to $\Z$ and get permutations $\tau_1,\dots,\tau_n$ satisfying the required conditions.        
\end{proof}   

\begin{proof}[proof of theorem \ref{free-dense}]
Let $j\in\{1,\dots,r\}$ be such that every word $w_i$ is not a conjugate of a power of $\tau_j$. If we set $\tau_j = \sigma$ then by lemma \ref{residual}, the set
\[
W_i = \left\{ (\tau_1,\dots,\tau_{j-1},\tau_{j+1},\dots,\tau_r)\in\S^{r-1}\ \left|\
\begin{array}{l}
\gen{\tau_1,\dots,\tau_r} \text{ is a dense, rank $r$ free subgroup}\\
\text{and }\gen{w_i\tuple{\tau}{r}} \text{ is non-discrete}
\end{array}
\right. \right\}
\] 
is residual and so $\bigcap_{i\in\N} W_i$ is residual and in particular, not empty.
\end{proof}

\section{Eventually Faithfull Sequences}\label{efs}

\begin{definition}
 Let $G$ and $H$ be groups. A sequence $\{f_n\}_{n=1}^{\infty}$ of homomorphisms from $G$ to $H$ is \emph{eventually faithfull} if for every $g\in G$ there exists $n_0\in\N$ such that $g\notin \text{ker}(f_n)$ for all $n\geq n_0$. 
\end{definition}
In order to prove the main theorem, we will need to produce eventually faithfull sequences of homomorphisms from surface groups to free groups. The following constructions apear in \cite{BGSS06} and \cite{BG09}. Let $\Gamma = \Gamma_{2r}$ be the surface group of genus $2r$ ($r\geq 1$). We have the following presentation for $\Gamma$
\[ \Gamma = \langle a_1,a'_1\dots,a_r,a'_r,b_1,b'_1,\dots,b_r,b'_r\ |\ [a_1,a'_1]\cdots [a_r,a'_r][b'_r,b_r]\cdots [b'_1,b_1]\rangle. \]
Let $x = [a_1,a'_1]\cdots [a_r,a'_r]$ and let $h : \Gamma\to\Gamma$ be the Dehn twist around $x$, i.e. 
\begin{align*}
&h(a_i)=a_i\qquad  h(b_i)=x b_i x^{-1}\\
&h(a'_i)=a'_i\qquad h(b'_i)=x b'_i x^{-1}
\end{align*}  
Let $F$ be the free group on $2r$ free generators $\{\phi_1,\phi'_1,\dots,\phi_r,\phi'_r\}$ and let $k : \Gamma\to F$ be the homomorphism defined by
\begin{align*}
& k(a_i) = k(b_i) = \phi_i\\
& k(a'_i) = k(b'_i) = \phi'_i.
\end{align*} 
Consider the map that folds the genus $2r$ surface that has $\Gamma$ as its fundamental group across the curve corresponding to $x$ (this curve seperates the surface into two equal parts). The image of this map is a surface of genus $r$ with one boundry component so it has $F$ as its fundamental group. $k$ is the homomorphism induced on the fundamental groups by this folding map (see figure 1). Denote $f_n = k\circ h^n$.
\begin{figure}[!hb]
\centering
\includegraphics[angle=90,scale=0.4]{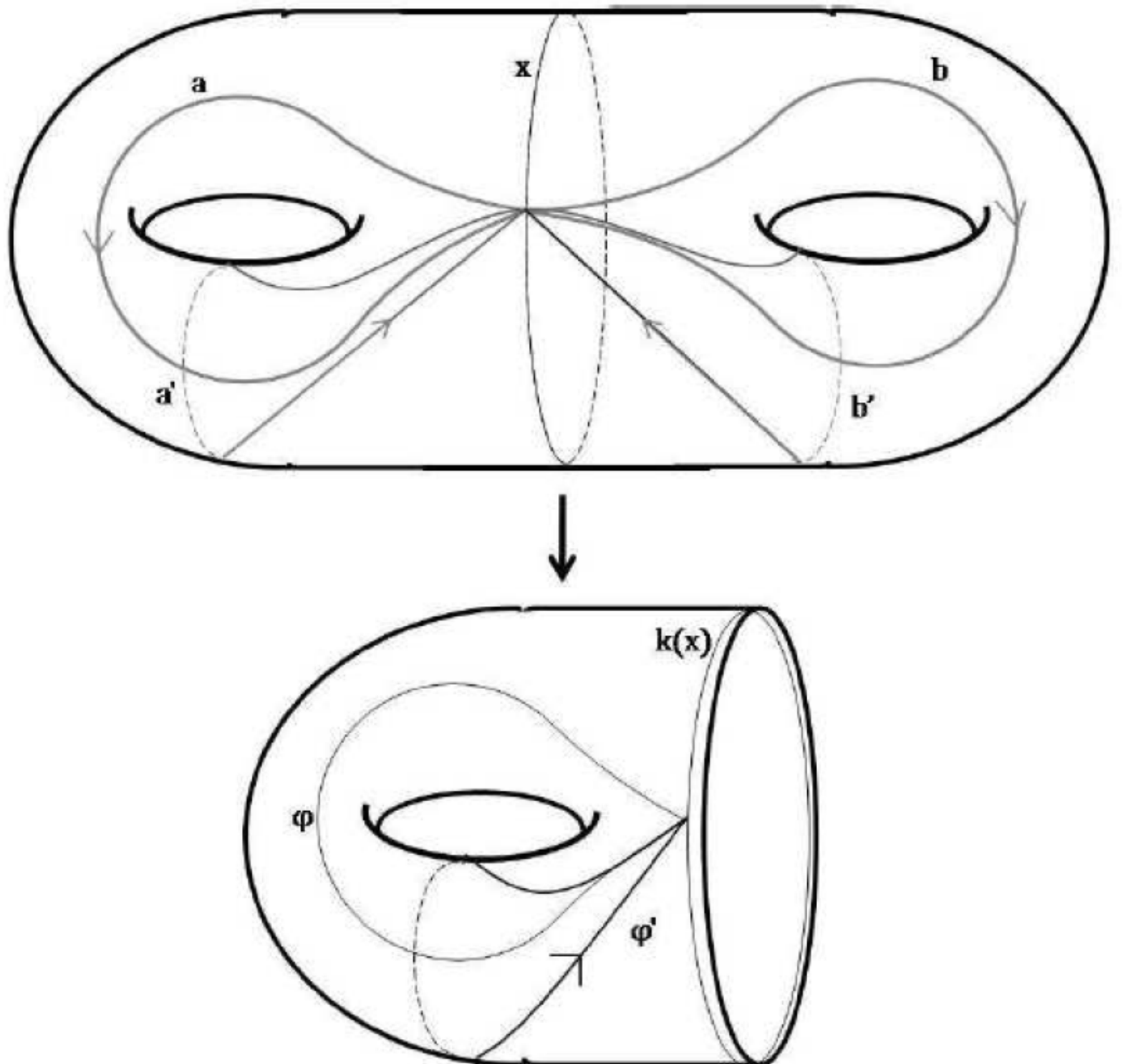}
\caption{}
\end{figure}
\begin{lem}[Breuillard,Gelander,Souto,Storm \cite{BGSS06}]\label{ef-even} 
The sequence $\{ f_n\}_{n=1}^{\infty}$ is eventually faithfull.
\end{lem}  

We now construct an eventually faithfull sequence of homomorphisms from an odd genus surface group into a free group. Let $\Gamma = \Gamma_{2r+1}$ be a surface group of genus $2r+1$ ($r\geq 1$) with the presentation
\[ \Gamma = \gen{a_1,a'_1,\dots,a_r,a'_r,b,b',c_1,c'_1,\dots,c_r,c'_r\ |\ [a_1,a'_1]\cdots [a_r,a'_r][b',b][c'_1,c_1]\cdots [c'_r,c_r]}\]
Denote $x=[a_1,a'_1]\cdots [a_r,a'_r]b'$ and let $F$ be the free group on $2r+1$ free generators $\{\phi_1,\phi'_1,\dots,\phi_r,\phi'_r,\tau\}$. Let $\delta:\Gamma\to\Gamma$ and $\zeta:\Gamma\to\Gamma$ be denh twists around $x$ and $b'$ respectively, that is
\begin{align*}
 \delta(a_i) &= a_i  & \zeta(a_i) &= a_i\\
 \delta(a'_i) &= a'_i & \zeta(a'_i) &= a'_i\\
 \delta(b) &= xb & \zeta(b) &= b(b')^{-1}\\
 \delta(b') &= b' & \zeta(b') &= b'\\
 \delta(c_i) &= xc_ix^{-1} & \zeta(c_i) &= c_i\\
 \delta(c'_i) &= xc'_ix^{-1} & \zeta(c'_i) &= c'_i\\
\end{align*}
\begin{figure}[!ht]
\centering
\includegraphics[angle=90,scale=0.4]{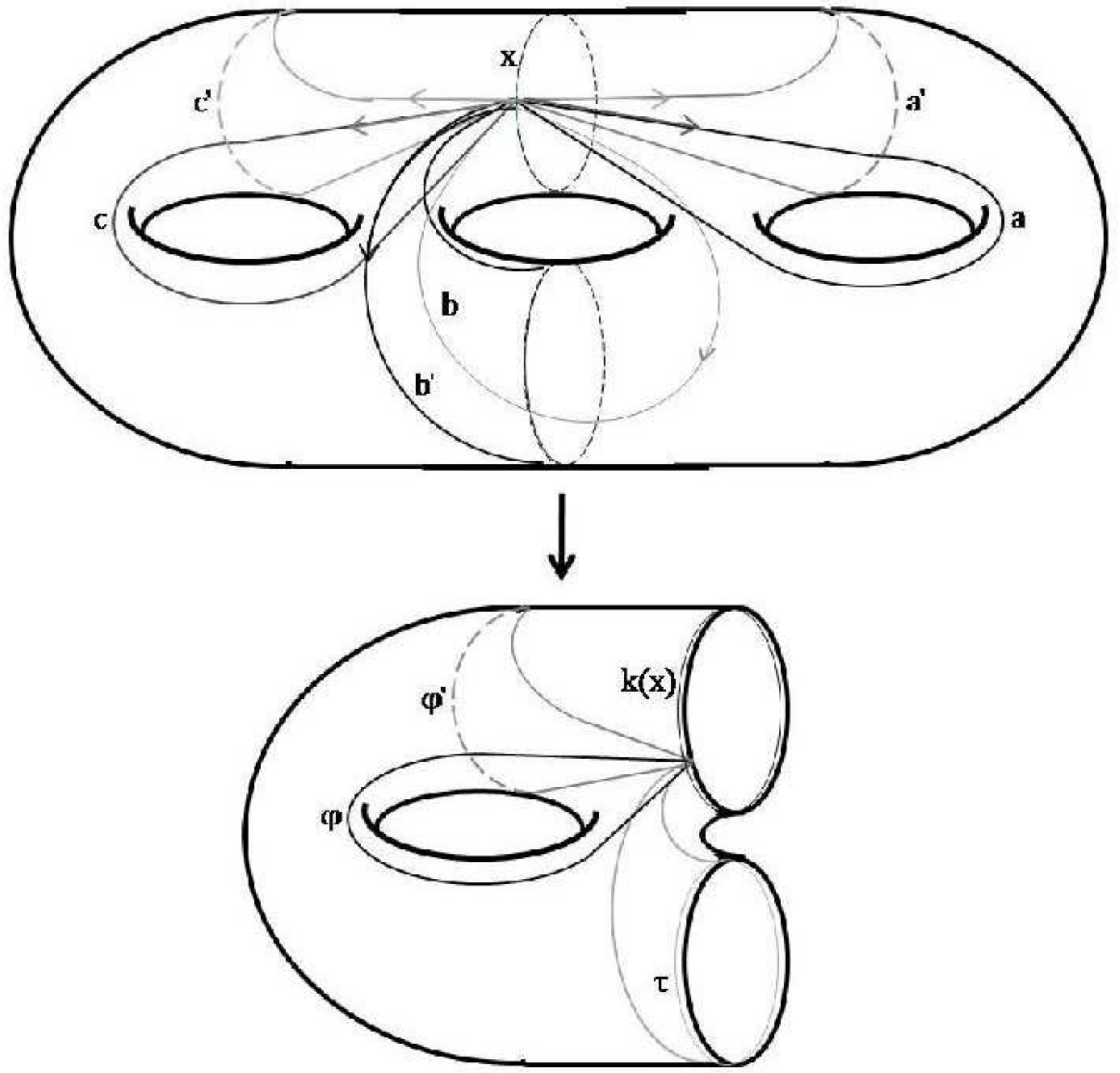}
\caption{}
\end{figure}
Notice that $\delta$ and $\zeta$ commute.

Let $k : \Gamma\to F$ be the map induced by folding the $2r+1$ surface across the curves $x$ and $b'$ (these curves seperate the surface into two surfaces of genus $r$ and two boundry components). Explicitely,
\begin{align*}
& k(a_i) = \phi_i\\
& k(a'_i) = \phi'_i\\
& k(b) = 1\\
& k(b'_i) = \tau\\
& k(c_i) = \phi_i\\
& k(c'_i) = \phi'_i
\end{align*}
(see figure 2). Finally, denote $\rho_n = k\circ (\delta\circ\zeta)^n$.
\begin{lem}[Barlev,Gelander \cite{BG09}]\label{ef-odd} 
The sequence $\{\rho_n\}_{n=1}^{\infty}$ is eventually faithfull.
\end{lem}   

\section{Proof of The Main Theorem}

We will need the following results.
\begin{lem}\label{nondisc2}
	Let $\phi,\psi\in\S$ be such that both $\gen{\phi}$ and $\gen{\psi}$ are non-discrete. Then
	$\gen{(\phi,\psi)}\leq\S^2$ is non-discrete.
\end{lem}
\begin{proof}
	We prove that every basic neighbourhood of $(1,1)$ contains a non-trivial element of
	$\gen{(\phi,\psi)}$. Let $m\in\N$. Let $\Delta = \bigcup_{i=1}^{\ell} \Delta_i$ and
	$\Gamma=\bigcup_{i=1}^k\Gamma_i$ be finite unions of orbits of $\gen{\phi}$ and 
	$\gen{\psi}$ respectively such that $[m]\subset\Delta$ and $[m]\subset\Gamma$. From
	proposition \ref{nondisc} we have that all the orbits of $\gen{\phi}$ and $\gen{\psi}$
	are finite and so we can define the number 
	\[
		n = \prod_{i=1}^{\ell} |\Delta_i|\cdot\prod_{i=1}^k |\Gamma_i|.
	\]
	Notice that every element of $\Delta$ is fixed by $\phi^n$ and every element of 
	$\Gamma$ is fixed by $\psi^n$ and in particular every $i\in [m]$ is fixed by $\phi^n$ 
	and $\psi^n$. From proposition \ref{nondisc} we also have that the lengths of the
	orbits of $\gen{\phi}$ and $\gen{\psi}$ are unbounded and in particular $\gen{\phi}$ 
	and $\gen{\psi}$ both have an orbit of length greater then $n$ and so $\phi^n$ and
	$\psi^n$ are non-trivial. Thus $(\phi,\psi)^n$ is a non-trivial element contained in
	the basic neighbourhood of $(1,1)$ defined by $[m]\times [m]$.
\end{proof}
\begin{lem}\label{lem1}
	Let G be a Hausdorff topological group and let $\gamma\in G$ such that $\gen{\gamma}$ is non-discrete. Then, for every $n_0\in\N$ the set
	$\{\gamma^n\ |\ n\geq n_0\}$ is dense in $\overline{\gen{\gamma}}$.
\end{lem}
\begin{proof}
	First we notice that since $\gen{\gamma}$ is non discrete then also $\overline{\gen{\gamma}}$ is non discrete. Let $U\subset 
	\overline{\gen{\gamma}}$ be open then we have $\gamma^m\in U$ for some $m\in\Z$. If $m\geq n_0$ we are done so assume $m<n_0$. Now, If we denote
	\begin{align*}
		U'&=U\gamma^{-m}\\
		U''&=U'\cap (U')^{-1}
	\end{align*}
	Then $U''$ is an open symmetric identity neighborhood and since $\overline{\gen{\gamma}}$ is non discrete, $U''$ is not finite. Now take
	\[
		\widetilde{U} = U''\smallsetminus\{\gamma^k\ : \ |k|< n_0-m\}.
	\] 
We have that $\widetilde{U}$ is open (since $G$ is Hausdorff) and non empty, thus there exits $n\in\Z$ such that $\gamma^n\in\widetilde{U}$. By the definition of $\widetilde{U}$ we have that $|n|\geq n_0-m$. Also, since $\widetilde{U}\subset U''$ and $U''$ is symmetric it follows that 					$\gamma^n,\gamma^{-n}\in U''\subset U'$. Hence, $\gamma^{n+m},\gamma^{-n+m}\in U$ and since either $n+m\geq n_0$ or $-n+m\geq n_0$ we are done. 
\end{proof}

\subsection{Proof of theorem \ref{main} for even genus}

By theorem \ref{free-dense} there exists a subgroup $F\leq\S$ such that $F$ is dense, free with $2r$ free generators $\phi_1,\phi'_1,\dots,\phi_r,\phi'_r\in\S$ and such that $\gen{[\phi_1,\phi'_1]\cdots [\phi_r,\phi'_r]}$ is non-discrete. Denote $\gamma = [\phi_1,\phi'_1]\cdots [\phi_r,\phi'_r]$ and $\Omega = \overline{\gen{\gamma}}$. Let $\Gamma = \Gamma_{2r}$ be a surface group of genus $2r$ ($r\geq 1$) with the presentation
\[ \Gamma = \langle a_1,a'_1\dots,a_r,a'_r,b_1,b'_1,\dots,b_r,b'_r\ |\ [a_1,a'_1]\cdots [a_r,a'_r][b'_r,b_r]\cdots [b'_1,b_1]\rangle. \]
Define for every $\omega\in\Omega$ a homomorphism $f_{\omega} : \Gamma\to\S$ by 
\begin{align*}
f_{\omega}(a_i) &= \phi_i\\
f_{\omega}(a'_i) &= \phi'_i\\
f_{\omega}(b_i) &= \omega\phi_i\omega^{-1}\\
f_{\omega}(b'_i) &= \omega\phi'_i\omega^{-1}.
\end{align*} 
Since every $\omega\in\Omega$ commutes with $\gamma$ this defines a homomorphism. Indeed, for every $\omega\in\Omega$: 
\begin{align*}
&f_{\omega}([a_1,a'_1]\cdots [a_r,a'_r][b'_r,b_r]\cdots [b'_1,b_1]) = \underbrace{[\phi_1,\phi'_1]\cdots [\phi_r,\phi'_r]}_{\gamma} \omega \underbrace{[\phi'_r,\phi_r] \cdots [\phi'_1,\phi_1]}_{\gamma^{-1}} \omega^{-1}=\\
&\gamma\omega\gamma^{-1}\omega^{-1} = 1.
\end{align*}
$\Omega$ is a completely metrizable space and every element of $\Omega$ corresponds to a homomorphism $\Gamma\to\S$ whose image contains $F$, hence the image is dense. We are left to show that at least one of those homorphisms is also faithfull. In fact, we show that $\chi = \{\omega\in\Omega\ |\ \text{$f_{\omega}$ is faithfull}\}$ is residual in $\Omega$. 

For every $g\in\Gamma\minus\{1\}$ denote $\chi_g = \{\omega\in\Omega\ |\ f_{\omega}(g) \neq 1\}$. Notice that 
\[\chi = \bigcap_{g\in\Gamma\minus\{1\}} \chi_g.\]
From lemma \ref{ef-even}, the sequence $\{f_{\gamma^n}\}_{n\in\N}$ is eventually faithfull and so for every $g\in\Gamma\minus\{1\}$ there exists $n_0\in\N$ such that $\gamma^n\in\chi_g$ for all $n\geq n_0$. By lemma \ref{lem1} we have that $\{\gamma^n\ |\ n\geq n_0\}$ is dense in $\Omega$ and so, $\chi_g$ is dense in $\Omega$. $\chi_g$ is also open as the inverse image of the the open set $\S\minus\{1\}$ under the continuous map $\omega\to f_{\omega}(g)$. This shows that $\chi$ is residual in $\Omega$. \qed    

\subsection{Proof of theorem \ref{main} for odd genus}

Let $\Gamma = \Gamma_{2r+1}$ be a surface group of genus $2r+1$ ($r\geq 1)$ with the presentation 
\[ \Gamma = \gen{a_1,a'_1,\dots,a_r,a'_r,b,b',c_1,c'_1,\dots,c_r,c'_r\ |\ [a_1,a'_1]\cdots [a_r,a'_r][b',b][c'_1,c_1]\cdots [c'_r,c_r]}\]
Let $F\leq\S$ be a dense, free subgroup with $2r+1$ free generators $\phi_1,\phi'_1,\dots,\phi_r,\phi'_r,\tau\in\S$ such that $\gen{\tau}$ and $\gen{[\phi_1,\phi'_1]\cdots [\phi_r,\phi'_r]\tau}$ are non-discrete (the existence of such a free subgroup is assured by theorem \ref{free-dense}).
Denote $\gamma = [\phi_1,\phi'_1]\cdots [\phi_r,\phi'_r]\tau$ and $\Omega = \overline{\gen{(\gamma,\tau)}}\subset \S^2$. For every $(\psi,\xi)\in\Omega$ we define a homomorphism $f_{(\psi,\xi)} : \Gamma\to\S$ by setting
\begin{align*}
f_{(\psi,\xi)}(a_i) & = \phi_i & f_{(\psi,\xi)}(b) &= \psi\xi^{-1} &f_{(\psi,\xi)}(c_i) &= \psi\phi_i\psi^{-1}\\
f_{(\psi,\xi)}(a'_i) &= \phi'_i& f_{(\psi,\xi)}(b') &= \tau & f_{(\psi,\xi)}(c'_i) &= \psi\phi'_i\psi^{-1}
\end{align*}
Since every $(\psi,\xi)\in\Omega$ commutes with $(\gamma,\tau)$ we have that $f_{(\psi,\xi)}$ is well defined as a homomorphism because
\begin{align*}
&f_{(\psi,\xi)}([a_1,a'_1]\cdots [a_r,a'_r][b',b][c'_1,c_1]\cdots [c'_r,c_r])=\\
&= [\phi_1,\phi'_1]\cdots [\phi_r,\phi'_r][\tau,\psi\xi^{-1}]\psi \underbrace{[\phi'_1,\phi_1]\cdots [\phi'_r,\phi_r]}_{=\tau\gamma^{-1}}\psi^{-1}=\\
&= \gamma\psi\xi^{-1}\tau^{-1}\xi\psi^{-1}\psi\tau\gamma^{-1}\psi^{-1} = \gamma\psi\xi^{-1}\tau^{-1}\xi\tau\gamma^{-1}\psi^{-1}=\\
&= \gamma\psi\xi^{-1}\xi\tau^{-1}\tau\gamma^{-1}\psi^{-1} = \gamma\psi\gamma^{-1}\psi^{-1} = 1.
\end{align*} 
Notice that $f_{(\gamma^n,\tau^n)} = \rho_n$ (as defined is Chapter \ref{efs}) for every $n\in\N$ and so by lemma \ref{ef-odd}, the sequence $f_{(\gamma^n,\tau^n)}$ is eventually faithfull.

The image of every homomorphism $f_{(\gamma^n,\tau^n)}$ contains $F$, hence the image is dense. Finally, we show that $\chi = \{ (\psi,\xi)\in\Omega\ |\ \text{$f_{(\psi,\xi)}$ is faithfull} \}$ is residual in the completely metrizable space $\Omega$ and in particular, $\chi\neq\varnothing$. For every $g\in\Gamma\minus\{1\}$ we denote $\chi_g = \{(\psi,\xi)\in\Omega\ |\ f_{(\psi,\xi)}(g)\neq 1 \}$. As in the previous section, $\chi_g$ is open as the inverse image of an open set under a continuous map. Since $f_{(\gamma^n,\tau^n)}$ is eventually faithfull there exits $n_0\in\N$ such that $f_{(\gamma^n,\tau^n)}\in\chi_g$ for all $n\geq n_0$. Since $\gen{\tau}$ and $\gen{\gamma}$ are non discrete we get from lemma \ref{nondisc2} that $\Omega$ is non-discrete and so by lemma \ref{lem1} we have that $\{(\gamma^n,\tau^n)\ |\ n\geq n_0\}$ is dense in $\Omega$. This shows that $\chi_g$ is dense, hence
\[ \chi = \bigcap_{g\in\Gamma\minus\{1\}} \chi_g \]
is residual. 

\bibliographystyle{amsplain}
{\bibliography{Daniel}}

\providecommand{\bysame}{\leavevmode\hbox to3em{\hrulefill}\thinspace}
\providecommand{\MR}{\relax\ifhmode\unskip\space\fi MR }
\providecommand{\MRhref}[2]{%
  \href{http://www.ams.org/mathscinet-getitem?mr=#1}{#2}
}
\providecommand{\href}[2]{#2}
\begin{thebibliography}{1}

\bibitem{BG09}
J.~Barlev and T.~Gelander, \emph{Compactifications and algebraic completions of
  limit groups}, arXiv:0904.3771v1 (2009).

\bibitem{D90}
J.~D. Dixon, \emph{Most finitely generated permutation groups are free}, Bull.
  London Math. Soc. \textbf{22} (1990), no.~3, 222--226.

\bibitem{BGSS06}
J.~Souto E.~Breuillard, T.~Gelander and P.~Storm, \emph{Dense embeddings of
  surface groups}, Geom. Topol \textbf{10} (2006), 1373--1389.

\bibitem{M77}
T.~P. McDonough, \emph{A permutation representation of a free group}, Quart. J.
  Math. Oxford Ser. (2) \textbf{22} (1977), no.~111, 353--356.

\end{thebibliography}
\end{document}